\documentclass[12pt,a4paper]{amsart}
\usepackage[utf8]{inputenc}
\usepackage[T1]{fontenc}
\usepackage[english]{babel}
\usepackage{amsmath}
\usepackage{amsfonts}
\usepackage{amssymb}
\usepackage{amsthm}
\usepackage{mathtools}
\usepackage{tabularx}
\usepackage{adjustbox}
\usepackage{graphicx}
\usepackage{tikz}
\usetikzlibrary{calc}
\usetikzlibrary{patterns.meta}

\usepackage{fullpage}
\usepackage{indentfirst}
\usepackage{fancyvrb}
\usepackage{comment}
\usepackage{lmodern}
\usepackage{xcolor}
\usepackage{listings}
\usepackage{url}
\usepackage{float}
\usepackage{diagbox}
\usepackage{hyperref}
\numberwithin{equation}{section}

\theoremstyle{plain}
\newtheorem{thm}{Theorem}[section]

\newtheorem{lm}[thm]{Lemma}

\theoremstyle{definition}
\newtheorem{df}[thm]{Definition}

\theoremstyle{remark}
\newtheorem{re}[thm]{Remark}

\DeclareRobustCommand{\svdots}{%
  \vbox{%
    \baselineskip=0.33333\normalbaselineskip
    \lineskiplimit=0pt
    \hbox{.}\hbox{.}\hbox{.}%
    \kern-0.2\baselineskip
  }%
}

\makeatletter
\let\@@pmod\pmod
\DeclareRobustCommand{\pmod}{\@ifstar\@pmods\@@pmod}
\def\@pmods#1{\mkern4mu({\operator@font mod}\mkern 6mu#1)}
\makeatother

\title{The asymptotic behavior of the rectangle partition function $p(m,n)$}

\author{Krystian Gajdzica}

\address{Theoretical Computer Science Department \\ Faculty of Mathematics and Computer Science\\ Jagiellonian University\\ Łojasiewicza 6\\ 30-348 Kraków\\ Poland}
\email{krystian.gajdzica@uj.edu.pl}

\author{Maciej Zakarczemny}

\address{Department of Applied Mathematics, Faculty of Computer Science and Mathematics, Cracow University of Technology, Warszawska 24, 31-155 Krak\'ow, Poland}

\email{maciej.zakarczemny@pk.edu.pl}



\begin{document}
\maketitle

\begin{abstract}
Let $p(m,n)$ denote the number of partitions of a rectangle $m\times n$ into integer-sided rectangular blocks, where two partitions are indistinguishable if they consist of the same multiset of blocks, regardless of their geometric arrangement. We present an elementary approach to show that, for every fixed positive integer $m$,
$$
\log p(m,n)=\pi\sqrt{\tfrac{2mH_m}{3}}\sqrt{n}+O(\log n),
\qquad \text{as }n\to\infty,
$$
where $H_m$ denotes the $m$-th harmonic number. This confirms a conjecture posed in~\cite{GZ} and generalizes the Hardy--Ramanujan formula for integer partitions.
\end{abstract}


\section{Introduction}

Let $m$ and $n$ be positive integers. A partition of the rectangle $m\times n$ is a multiset of integer-sided rectangular blocks that can be arranged, with sides parallel to the coordinate axes, in such a way that they cover exactly the rectangle $m\times n$ and their interiors do not intersect. Blocks of sizes $a\times b$ and $b\times a$ are indistinguishable, and two partitions of a rectangle are considered to be the same if they are equal as multisets (the geometric arrangement of the blocks is irrelevant). The total number of partitions of the rectangle $m\times n$ is denoted by $p(m,n)$. For the precise definitions of a partition of a rectangle and the rectangle partition function, we refer the reader to Section~\ref{sec: preliminaries}. The above notion was introduced in \cite{GVZ} as a two-dimensional generalization of the classical partition function $p(n)$. Indeed, every partition of the rectangle $1\times n$ has the form $1\times \lambda_1,1\times \lambda_2,\ldots,1\times \lambda_r,$ and directly corresponds to the integer partition $\lambda_1,\lambda_2,\ldots,\lambda_r$ of $n$. Hence, $p(1,n)=p(n)$, and the Hardy--Ramanujan formula \cite{HardyRamanujan} gives

$$
\log p(1,n)=\pi\sqrt{\tfrac{2n}{3}}+O(\log n).
$$
For $m=2$, the asymptotic behavior of $p(2,n)$ was recently determined in \cite{GVZ}; in particular,
$$
\log p(2,n)=\pi\sqrt{2n}+O(\log n).
$$
For $m=3$, it was proved in \cite{GZ} that
$$
\log p(3,n)=\pi\sqrt{\tfrac{11n}{3}}+O(\log n).
$$
Let $H_m=\sum_{a=1}^{m}\tfrac1a$ denote the $m$-th harmonic number. The leading term in each of the three asymptotic formulae above agrees with the leading term on the right-hand side of
\begin{equation}\label{eq: conjecture}
\log p(m,n)=\pi\sqrt{\tfrac{2mH_m}{3}}\sqrt{n}+o\left(\sqrt n\right),
\qquad \text{as }n\to\infty,
\end{equation}
which was conjectured in \cite{GZ} for every fixed positive integer $m$. The purpose of this paper is to establish the conjecture in the following stronger form.

\begin{thm}\label{thm: main}
Let $m$ be a fixed positive integer. As $n\to\infty$, we have
$$
\log p(m,n)=\pi\sqrt{\tfrac{2mH_m}{3}}\sqrt{n}+O(\log n).
$$
Moreover, $p(m,n)\le\exp\left(\pi\sqrt{\tfrac{2mH_m}{3}}\sqrt n\right)$ for all $n\geq m$.
\end{thm}

For the reader's convenience, we outline the idea of the proof of Theorem \ref{thm: main}. It is divided into two separate parts, concerning the derivation of the upper and lower bounds for $p(m,n)$.\\
\indent The proof of the upper bound compares $p(m,n)$ with the number of multisets of blocks of total area $mn$ whose smaller sides do not exceed $m$ (there is no requirement that such a multiset of blocks has to form a tiling). The generating function of the latter quantity factors into shifted Euler products, and an elementary estimate of these products provides the upper bound without any error term.\\
\indent The proof of the lower bound is via a construction. For every
$a\in\{1,2,\ldots,m\}$, we independently choose an integer partition
$\lambda_a$ from a suitable family and associate with each of its parts $b$
a block of size $a\times b$. Thus, the index $a$ specifies the height of
the corresponding blocks, while their widths are the parts of $\lambda_a$.
For every $a\ge 2$, we first reserve a column of height $m$ containing
$\lfloor m/a\rfloor$ horizontal strips of height $a$, completed by strips
of height $1$ to the full height $m$. These columns are placed side by side,
and the blocks associated with $\lambda_a$ are packed into the corresponding
strips of height $a$. The~unused parts of these strips, the strips of height
$1$ already present in the columns, and the part of the rectangle remaining
to the right of all the columns are then divided into horizontal strips of
height $1$. The blocks associated with $\lambda_1$ are packed into these
strips, and every remaining unit of area is filled with a square
$1\times1$. The widths of the columns are estimated by the elementary
inequality
$
m\sum_{a=2}^{m}\frac{1}{a^2\lfloor m/a\rfloor}<H_m,
$
proved in Lemma~\ref{lm: harmonic}, which guarantees that all the columns
fit side by side inside the rectangle $m\times n$. The~multiset of the
selected blocks determines the partitions
$\lambda_1,\lambda_2,\ldots,\lambda_m$ uniquely, so the number of
constructed partitions of the rectangle is the product of the cardinalities
of the families, and this product has the required order of magnitude.\\
The paper is self-contained. Section~\ref{sec: preliminaries} collects some notation and definitions, Section~\ref{sec: lemmas} contains the auxiliary lemmas, and Section~\ref{sec: proof} is devoted to the proof of Theorem~\ref{thm: main}.

\section{Preliminaries}\label{sec: preliminaries}

We write $\mathbb{Z}^+$ and $\mathbb{R}$ for the set of positive integers and the set of real numbers, respectively. Throughout the paper, $m$ is a fixed positive integer, and the implied constants in the big $O$ notation may depend on $m$. A partition of a non-negative integer $t$ is a non-increasing sequence $\lambda=(\lambda_1,\lambda_2,\ldots,\lambda_r)$ of positive integers such that
$
|\lambda|:=\lambda_1+\lambda_2+\cdots+\lambda_r=t.
$
The elements $\lambda_i$ are called parts of the partition $\lambda$. The number of all partitions of $t$ is denoted by $p(t)$. In particular, we have that $p(t)=0$ if $t<0$ and $p(0)=1$ with the empty partition. Clearly, the map that adds a part equal to $1$ at the end of a partition is an injection from the set of partitions of $t-1$ into the set of partitions of $t$, so the sequence $p(t)$ is weakly increasing. Further, for a positive integer $\mu$, let $p_{\le\mu}(t)$ denote the number of partitions of $t$ whose parts are at most $\mu$. We also set
$
\mu(t):=\left\lceil2\sqrt{t}\log t\right\rceil
$
for every integer $t\geq2$. Let us notice that the Hardy--Ramanujan formula \cite{HardyRamanujan},
$$
p(t)=\frac{1}{4\sqrt3\,t}\exp\left(\pi\sqrt{\tfrac{2t}{3}}\right)\left(1+O\left(t^{-1/2}\right)\right),
$$
ensures that
\begin{equation}\label{eq: HR}
\log p(t)=\pi\sqrt{\tfrac{2t}{3}}-\log t+O(1).
\end{equation}

We now recall the definition of a partition of a rectangle from \cite{GVZ}. As it was indicated earlier, all rectangles are assumed to have sides parallel to the coordinate axes.

\begin{df}\label{df: rectangle partition}
Let $R=[0,m]\times[0,n]\subset\mathbb{R}^2$, where $m,n\in\mathbb{Z}^+$. A~partition of the rectangle $R$ is a~finite multiset $\mathcal{P}=\{R_i\}_{i\in I}$ of integer-sided rectangles such that:
\begin{enumerate}
\item rectangles of sizes $a\times b$ and $b\times a$ are indistinguishable;
\item all the rectangles from $\mathcal{P}$ can be arranged in such a~way that they cover exactly the area of the rectangle $R$ and their interiors do not intersect.
\end{enumerate}
Two partitions of a rectangle are considered to be the same if and only if they are equal as multisets. The number of partitions of a rectangle of size $m\times n$ is denoted by $p(m,n)$.
\end{df}

We say that a rectangle of size $a\times b$ with $a\le b$ is written in the canonical form. If $n\geq m$, then, in any arrangement from Definition~\ref{df: rectangle partition}, both sides of every rectangle from $\mathcal P$ are at most $n$ and one of them is at most $m$. Hence, the smaller side of every rectangle from $\mathcal P$ is always at most $m$ in such a setting.

\section{Lemmata}\label{sec: lemmas}

The first lemma contains elementary estimates of finite sums of reciprocals of positive integers.

\begin{lm}\label{lm: sums}
For all $x\in(0,1)$, we have $\log(1+x)<x<-\log(1-x)$. Moreover, for all integers $1\le u<v$,
$$
\log\frac{v+1}{u+1}
<
\sum_{j=u+1}^{v}\frac1j
<
\log\frac{v}{u}
\qquad\text{and}\qquad
\sum_{j=u+1}^{v}\frac{1}{j^2}
<
\frac{2}{2u+1}-\frac{2}{2v+1}.
$$
\end{lm}

\begin{proof}
For the first part, it suffices to consider the functions $x-\log(1+x)$ and $-\log(1-x)-x$, and their derivatives. To deduce the second one, we apply both inequalities with $x=\tfrac1j$ to get
$$
\log\frac{j+1}{j}<\frac1j<\log\frac{j}{j-1}
$$
for every integer $j\geq2$. Summing over $j\in\{u+1,u+2,\ldots,v\}$ and telescoping gives the required estimates. Finally, for every integer $j\geq1$, we have that
$$
\frac{1}{j^2}<\frac{1}{j^2-\tfrac14}=\frac{2}{2j-1}-\frac{2}{2j+1}.
$$
Summing over $j\in\{u+1,u+2,\ldots,v\}$ and telescoping completes the proof.
\end{proof}

The second lemma presents an upper bound for some infinite series of logarithms.

\begin{lm}\label{lm: euler product}
For every real $u>0$, we have
$$
\sum_{j=1}^{\infty}\log\frac{1}{1-e^{-ju}}
\le
\frac{\pi^2}{6u}.
$$
\end{lm}

\begin{proof}
All the terms below are positive, so the order of summation can be interchanged. Expanding the logarithm into its Maclaurin series and summing the resulting geometric series gives us that
$$
\sum_{j=1}^{\infty}\log\frac{1}{1-e^{-ju}}
=
\sum_{j=1}^{\infty}\sum_{k=1}^{\infty}\frac{e^{-jku}}{k}
=
\sum_{k=1}^{\infty}\frac1k\cdot\frac{e^{-ku}}{1-e^{-ku}}
=
\sum_{k=1}^{\infty}\frac{1}{k\left(e^{ku}-1\right)}.
$$
Since $e^x>1+x$ for $x>0$, every term is smaller than $\tfrac{1}{k^2u}$. The claim now follows from the well-known Euler's identity $\sum_{k=1}^{\infty}\tfrac{1}{k^2}=\tfrac{\pi^2}{6}$.
\end{proof}

The third lemma estimates  $p_{\le\mu(t)}(t)$ in terms of the partition function $p(t)$.

\begin{lm}\label{lm: largest part}
We have $p_{\le\mu(t)}(t)\geq\tfrac12p(t)$ for all sufficiently large $t$.
\end{lm}

\begin{proof}
Let $E(t)$ denote the number of partitions of $t$ with at least one part greater than $\mu(t)$. Removing one copy of the largest part $b$ is, for every fixed $b$, an injection into the set of partitions of $t-b$. Since the sequence $p$ is weakly increasing, we have that
$$
E(t)\le\sum_{b=\mu(t)+1}^{t}p(t-b)\le t\,p\bigl(t-\mu(t)\bigr).
$$
For all sufficiently large $t$, we have that $\mu(t)\le\tfrac t2$,
$$\log\tfrac{t}{t-\mu(t)}
\le
\log2\qquad\text{and}\qquad
\sqrt{t}-\sqrt{t-\mu(t)}
=
\tfrac{\mu(t)}{\sqrt t+\sqrt{t-\mu(t)}}
\geq
\tfrac{\mu(t)}{2\sqrt t}
\geq
\log t.
$$
Therefore, for all large values of $t$, Equality \eqref{eq: HR} implies that
\begin{align*}
\log\frac{t\,p(t-\mu(t))}{p(t)}
&=
\log t+\log\frac{t}{t-\mu(t)}
-\pi\sqrt{\tfrac23}\left(\sqrt t-\sqrt{t-\mu(t)}\right)+O(1)\\
&\le
-\left(\pi\sqrt{\tfrac23}-1\right)\log t+O(1).
\end{align*}
Since $\pi\sqrt{2/3}>2$, the right-hand side tends to $-\infty$. In particular, $E(t)\le\tfrac12p(t)$ for all sufficiently large $t$. The equality $p_{\le\mu(t)}(t)=p(t)-E(t)$ concludes the proof.
\end{proof}

To deal with the lower bound for $p(m,n)$, we also need to introduce a special partition statistic together with its fundamental properties.

\begin{lm}\label{lm: families}
Let $A\geq2$ and $K\geq1$ be fixed integers. For a sufficiently large positive integer $\ell$, define
$$
\sigma_\ell=\left\lfloor\ell-(2K+1)\sqrt\ell\log\ell\right\rfloor,
\qquad
\nu_\ell=\mu(\sigma_\ell)=\left\lceil2\sqrt{\sigma_\ell}\log\sigma_\ell\right\rceil,
$$
and let $\mathcal A(A,\ell)$ be the set of all partitions $\lambda$ such that $|\lambda|\le\sigma_\ell$ and every part of $\lambda$ belongs to the set $\{A,A+1,\ldots,\nu_\ell\}$. Then
$$
\sigma_\ell\le\ell-K\nu_\ell
$$
for all sufficiently large $\ell$, and
$$
\log|\mathcal A(A,\ell)|
=
\pi\sqrt{\tfrac23}\sqrt\ell+O(\log\ell),
$$
where the implied constant may depend on $A$ and $K$.
\end{lm}

\begin{proof}
We first note that
$
\sigma_\ell\sim\ell,
\,\,
\nu_\ell\sim2\sqrt{\ell}\log\ell.
$
Since $\sigma_\ell\le\ell$, we have $\nu_\ell\le2\sqrt\ell\log\ell+1$. Hence,
$$
K\nu_\ell\le2K\sqrt\ell\log\ell+K\le(2K+1)\sqrt\ell\log\ell
$$
for all sufficiently large $\ell$. Moreover, $\sigma_\ell\le\ell-(2K+1)\sqrt\ell\log\ell$ by the definition of $\sigma_\ell$. 
Adding the above inequalities side by side gives $\sigma_\ell+K\nu_\ell\le\ell$, and, in consequence, $\sigma_\ell\le\ell-K\nu_\ell$ for all sufficiently large $\ell$.

We now estimate $|\mathcal A(A,\ell)|$ from below. Consider the set of all partitions of $\sigma_\ell$ whose parts are at most $\nu_\ell$. Lemma~\ref{lm: largest part} guarantees that its cardinality is $p_{\le\nu_\ell}(\sigma_\ell)\geq\tfrac12p(\sigma_\ell)$ for all sufficiently large $\ell$. Let $\lambda$ be an arbitrary element of that set of partitions. If we delete all the parts smaller than $A$ from $\lambda$, then we get a partition $\tilde{\lambda}$ belonging to $\mathcal A(A,\ell)$.
The~deleted parts form a partition of $\sigma_\ell-|\tilde{\lambda}|$ into parts smaller than $A$. Such a~partition is determined by the multiplicities of the parts $1,2,\ldots,A-1$, and each of these multiplicities is at most $\sigma_\ell$. Therefore, for each fixed $\tilde{\lambda}\in\mathcal A(A,\ell)$, there are at most $(\sigma_\ell+1)^{A-1}$ partitions $\lambda$ that can produce $\tilde{\lambda}$ after deleting all parts smaller than $A$. Hence,
\begin{align}\label{in: A(A,l)>}
|\mathcal A(A,\ell)|
\ge
\frac{p_{\le\nu_\ell}(\sigma_\ell)}{(\sigma_\ell+1)^{A-1}}
\ge
\frac{p(\sigma_\ell)}{2(\sigma_\ell+1)^{A-1}}
\end{align}
for all large values of $\ell$. On the other hand, since the partition function $p(n)$ is weakly increasing,
\begin{align}\label{in: A(A,l)<}
|\mathcal A(A,\ell)|
\le
\sum_{j=0}^{\sigma_\ell}p(j)
\le
(\sigma_\ell+1)p(\sigma_\ell).
\end{align}
Consequently, Inequalities \eqref{in: A(A,l)>} and \eqref{in: A(A,l)<}  point out that $\log|\mathcal A(A,\ell)|=\log p(\sigma_\ell)+O\left(\log\left(\sigma_\ell+1\right)\right)$. Since $\ell-\sigma_\ell\le(2K+1)\sqrt\ell\log\ell+1$, we have
$$
\sqrt\ell-\sqrt{\sigma_\ell}
=
\frac{\ell-\sigma_\ell}{\sqrt\ell+\sqrt{\sigma_\ell}}
=
O(\log\ell),
$$
so $\sqrt{\sigma_\ell}=\sqrt\ell+O(\log\ell)$. Therefore, we obtain that $\log|\mathcal A(A,\ell)|=\log p(\sigma_\ell)+O(\log\ell)$. The claim now follows from Equality \eqref{eq: HR}.
\end{proof}
\begin{re}
In Lemma~\ref{lm: families}, the implied constant in the term $O(\log\ell)$ may depend on $A$ and $K$. In the proof of Theorem~\ref{thm: main}, we only use $K=m^2$ and $A\in\{2,3,\ldots,m\}$, where $m$ is fixed. Hence, the resulting implied constant depends only on $m$.
\end{re}
The succeeding lemma is the packing argument used in the construction in the proof of Theorem \ref{thm: main}.

\begin{lm}\label{lm: strips}
Let $k$, $\mu$ and $s$ be positive integers, and let $S_1,S_2,\ldots,S_k$ be pairwise disjoint horizontal strips of common height $s$ and integer widths $c_1,c_2,\ldots,c_k$ such that $c_j\geq\mu$ for every $j\in\{1,2,\ldots,k\}$. Each multiset of integer-sided rectangles of height $s$, whose widths are at most $\mu$ and whose total width is at most
$$
c_1+c_2+\cdots+c_k-k\mu,
$$
can be packed into $S_1\cup S_2\cup\cdots\cup S_k$ in such a way that the interiors of the rectangles do not intersect. Moreover, the packing can be chosen so that, in every strip, the rectangles are placed consecutively, starting from the left side of the strip. In particular, the unused part of every strip is a rectangle of height $s$.
\end{lm}

\begin{proof}
Order the rectangles arbitrarily and place them greedily. We proceed through the strips $S_1,S_2,\ldots,S_k$ one by one. In the current strip, the rectangles are placed consecutively, each one adjacent to the previously placed one, and the first one adjacent to the left side of the strip. When the next rectangle does not fit into the current strip, we pass to the next strip and never return. Since the width of every rectangle is at most $\mu\le c_j$, a strip $S_j$ is abandoned only when its occupied width exceeds $c_j-\mu$.

Suppose that some rectangle cannot be placed after all of the $k$ strips have been used. Then the occupied width of $S_j$ exceeds $c_j-\mu$ for every $j\in\{1,2,\ldots,k\}$, so the total width of the placed rectangles exceeds $c_1+c_2+\cdots+c_k-k\mu$, which contradicts the assumption on the total width. Hence, all the rectangles are placed, and the packing has the required form.
\end{proof}

For the construction in the proof of Theorem \ref{thm: main}, we also need the following lemma, which was mentioned in Introduction.

\begin{lm}\label{lm: harmonic}
For an integer $m\geq2$ and $2\le a\le m$, put $q_a=\lfloor m/a\rfloor$ and $r_a=m-aq_a$. Then
$$
m\sum_{a=2}^{m}\frac{1}{a^2q_a}<H_m.
$$
\end{lm}

\begin{proof}
Since $m=aq_a+r_a$, we have
$$
\frac{m}{a^2q_a}=\frac1a+\frac{r_a}{a^2q_a}.
$$
If we sum over $2\le a\le m$ and use $\sum_{a=2}^{m}\tfrac1a=H_m-1$, we will get that the assertion from the statement is equivalent to the inequality 
$$
\sum_{a=2}^{m}\frac{r_a}{a^2q_a}<1.
$$
Let us denote the left hand side of the above inequality by $\Sigma_m$. We have $\Sigma_2=0$ and $\Sigma_3=\tfrac14$, so let $m\geq4$ and set $h=\lfloor m/2\rfloor\geq2$.

If $2\le a\le h$, then observe that the inequality $r_a(m-a)<a(m-r_a)$ is equivalent to $r_am<am$, which is valid as $r_a<a$. Hence,
$$
\frac{r_a}{a^2q_a}=\frac{r_a}{a(m-r_a)}<\frac{1}{m-a}.
$$

If $h<a\le m$, then $a>\tfrac m2$, so $q_a=1$ and $r_a=m-a$. Thus,
$$
\frac{r_a}{a^2q_a}=\frac{m-a}{a^2}=\frac{m}{a^2}-\frac1a.
$$

Assume first that $m=2h$. The substitution $j=m-a$ and Lemma~\ref{lm: sums} give
$$
\sum_{a=2}^{h}\frac{1}{m-a}
=
\sum_{j=h}^{2h-2}\frac1j
<
\log\frac{2h-2}{h-1}
=
\log2.
$$
Moreover, by Lemma~\ref{lm: sums},
$$
2h\sum_{a=h+1}^{2h}\frac{1}{a^2}
<
2h\left(\frac{2}{2h+1}-\frac{2}{4h+1}\right)
=
\frac{8h^2}{(2h+1)(4h+1)}
=
1-\frac{6h+1}{(2h+1)(4h+1)}
$$
and
$$
\sum_{a=h+1}^{2h}\frac1a
>
\log\frac{2h+1}{h+1}.
$$
Combining all of the above estimates together, we obtain
$$
\Sigma_m
<
\log2
+1-\frac{6h+1}{(2h+1)(4h+1)}
-\log\frac{2h+1}{h+1}
=
1-\frac{6h+1}{(2h+1)(4h+1)}
+\log\frac{2h+2}{2h+1}.
$$
By Lemma~\ref{lm: sums},
$$
\log\frac{2h+2}{2h+1}
<
\frac{1}{2h+1}
=
\frac{4h+1}{(2h+1)(4h+1)}
\le
\frac{6h+1}{(2h+1)(4h+1)},
$$
and, thus, $\Sigma_m<1$.

Assume now that $m=2h+1$. Analogously, we get that
$$
\sum_{a=2}^{h}\frac{1}{m-a}
=
\sum_{j=h+1}^{2h-1}\frac1j
<
\log\frac{2h-1}{h},
$$
$$
(2h+1)\sum_{a=h+1}^{2h+1}\frac{1}{a^2}
<
(2h+1)\left(\frac{2}{2h+1}-\frac{2}{4h+3}\right)
=
\frac{4h+4}{4h+3}
=
1+\frac{1}{4h+3}
$$
and
$$
\sum_{a=h+1}^{2h+1}\frac1a
>
\log\frac{2h+2}{h+1}
=
\log2.
$$
Combining these estimates all together, we deduce that
$$
\Sigma_m
<
\log\frac{2h-1}{h}
+1+\frac{1}{4h+3}
-\log2
=
1+\frac{1}{4h+3}
-\log\frac{2h}{2h-1}.
$$
Once again, Lemma~\ref{lm: sums} asserts that
$$
\log\frac{2h}{2h-1}
>
\frac{1}{2h}
>
\frac{1}{4h+3},
$$
and therefore $\Sigma_m<1$ in both cases, which completes the proof.
\end{proof}

We are now in the position to proceed to the main part of the paper.

\section{Proof of Theorem \ref{thm: main}}\label{sec: proof}

 As indicated in Introduction, the proof of Theorem \ref{thm: main} is divided into two parts. We first establish the upper bound for $p(m, n)$, and then investigate the corresponding lower bound.
 
\begin{proof}[Proof of Theorem \ref{thm: main}]
At first, let us assume that $m\geq2$ is fixed.

\medskip
\noindent\emph{Upper bound.}
For a non-negative integer $\ell$, let $U_m(\ell)$ denote the number of multisets of integer-sided rectangles written in the canonical form $a\times b$, where $1\le a\le m$ and $a\le b$, whose total area is $\ell$. The multiset does not have to form a tiling. Let $n\geq m$. As observed in Section~\ref{sec: preliminaries}, the smaller side of every rectangle in a partition of the rectangle $m\times n$ is at most $m$, and the total area of the rectangles equals $mn$. Hence, every partition of the rectangle $m\times n$ is a multiset counted by $U_m(mn)$, and
$$
p(m,n)\le U_m(mn).
$$
Observe that every pair $(a,b)$ with $1\le a\le m$ and $a\le b$ corresponds to exactly one canonical rectangle, so the generating function of the sequence $U_m(\ell)$ is
$$
F_m(x)
=
\sum_{\ell=0}^{\infty}U_m(\ell)x^{\ell}
=
\prod_{a=1}^{m}\prod_{b=a}^{\infty}\frac{1}{1-x^{ab}},
\qquad
\text{for }0<x<1.
$$
Every omitted factor satisfies $\tfrac{1}{1-x^{ab}}\geq1$. Hence, Lemma~\ref{lm: euler product} applied with $u=at$ guarantees that
$$
\log F_m\left(e^{-t}\right)
\le
\sum_{a=1}^{m}\sum_{b=1}^{\infty}\log\frac{1}{1-e^{-abt}}
\le
\sum_{a=1}^{m}\frac{\pi^2}{6at}
=
\frac{\pi^2H_m}{6t}
$$
for every $t>0$. Since all the coefficients of $F_m$ are non-negative, we have $U_m(mn)e^{-tmn}\le F_m\left(e^{-t}\right)$, and hence
$$
\log p(m,n)
\le
tmn+\frac{\pi^2H_m}{6t}.
$$
Putting $t=\pi\sqrt{\tfrac{H_m}{6mn}}$, we obtain $tmn=\tfrac{\pi^2H_m}{6t}=\pi\sqrt{\tfrac{mH_mn}{6}}$, and therefore
\begin{equation}\label{eq: UB}
\log p(m,n)
\le
\pi\sqrt{\tfrac{2mH_m}{3}}\sqrt n
\end{equation}
for all $n\geq m$.

\medskip
\noindent\emph{Lower bound.}
Put $K=m^2$, and let $\sigma_\ell$, $\nu_\ell$ and $\mathcal A(A,\ell)$ be defined as in Lemma~\ref{lm: families}. For $1\le a\le m$, we set
$$
T_{a,n}=\left\lfloor\frac{mn}{H_ma^2}\right\rfloor,
\qquad
\nu_{a,n}=\nu_{T_{a,n}},
$$
and, for $2\le a\le m$,
$$
q_a=\left\lfloor\frac{m}{a}\right\rfloor,
\qquad
c_{a,n}=\left\lceil\frac{T_{a,n}}{q_a}\right\rceil+\nu_{a,n}.
$$
Since $m$ is fixed, $T_{a,n}$ tends to infinity as $n\to\infty$ for every $a$, so all these parameters are well defined for all sufficiently large values of  $n$. Furthermore, we have that
$$
\nu_{a,n}\le2\sqrt{T_{a,n}}\log T_{a,n}+1=O\left(\sqrt n\log n\right),
$$
as $\sigma_\ell\leq\ell$, and $a\le m\le\nu_{a,n}$ for all sufficiently large $n$. Choose independently a~partition $\lambda_1\in\mathcal A(2,T_{1,n})$ and, for every $a\in\{2,3,\ldots,m\}$, a~partition $\lambda_a\in\mathcal A(a,T_{a,n})$. Replace every part $b$ of $\lambda_a$ by a rectangle $a\times b$; this is admissible, because $a\le b\le\nu_{a,n}$ for $a\geq2$, and $2\le b\le\nu_{1,n}$ for $a=1$.

For every $a\in\{2,3,\ldots,m\}$, reserve a column of height $m$ and width $c_{a,n}$, consisting of $q_a$ horizontal strips of height $a$, stacked one on the top of the other, and of $m-aq_a$ horizontal strips of height $1$ above them. The columns are placed side by side, starting from the left side of the rectangle $m\times n$. Since $m$ is fixed,
$$
\sum_{a=2}^{m}c_{a,n}
\le
\sum_{a=2}^{m}\frac{T_{a,n}}{q_a}+(m-1)+\sum_{a=2}^{m}\nu_{a,n}
\le
\frac{mn}{H_m}\sum_{a=2}^{m}\frac{1}{a^2q_a}
+O\left(\sqrt n\log n\right).
$$
Lemma~\ref{lm: harmonic} asserts that
$$
\theta_m:=\frac{m}{H_m}\sum_{a=2}^{m}\frac{1}{a^2q_a}<1.
$$
Hence,
$$
\sum_{a=2}^{m}c_{a,n}
\le
\theta_mn+O\left(\sqrt n\log n\right)
\le
n
$$
for all sufficiently large $n$, so all the columns fit side by side inside the rectangle $m\times n$. The part of the rectangle $m\times n$ to the right of all of the aforementioned columns is a rectangle of height $m$.

Fix $a\in\{2,3,\ldots,m\}$. The rectangles corresponding to $\lambda_a$ have common height $a$, their widths are at most $\nu_{a,n}$, and their total width satisfies
$$
|\lambda_a|
\le
\sigma_{T_{a,n}}
\le
T_{a,n}
\le
q_a\left\lceil\frac{T_{a,n}}{q_a}\right\rceil
=
q_a\left(c_{a,n}-\nu_{a,n}\right).
$$
Lemma~\ref{lm: strips}, applied with $k=q_a$, $\mu=\nu_{a,n}$ and $s=a$, maintains that these rectangles can be packed into the $q_a$ reserved strips of height $a$ and width $c_{a,n}$, in such a way that the unused part of every strip is a rectangle of height $a$.

We now describe the part of the rectangle $m\times n$ that remains uncovered. Divide the unused rectangle of height $a$ at the end of every strip into $a$ horizontal strips of height $1$, and divide the rectangle to the right of all the columns into $m$ horizontal strips of height $1$. Every column contributes at most $aq_a+(m-aq_a)=m$ strips of height $1$, so, together with the strips obtained from the rectangle on the right, the uncovered part of the rectangle $m\times n$ is a disjoint union of at most
$$
m(m-1)+m=m^2
$$
horizontal strips of height $1$ with integer widths. The total width of these strips equals their total area, and
$$
mn-\sum_{a=2}^{m}a|\lambda_a|
\geq
mn-\sum_{a=2}^{m}aT_{a,n}
\geq
mn-\frac{mn}{H_m}\sum_{a=2}^{m}\frac1a
=
\frac{mn}{H_m}
\geq
T_{1,n}.
$$
Remove the strips of width smaller than $\nu_{1,n}$, and let $c_1,c_2,\ldots,c_r$ denote the widths of the remaining strips, so that $c_j\geq\nu_{1,n}$ for every $1\leq j\leq r$. The total number of both the removed strips and the remaining ones is at most $m^2$. Thus,
$$
\sum_{j=1}^{r}c_j-r\nu_{1,n}
\geq
T_{1,n}-m^2\nu_{1,n}.
$$
Lemma~\ref{lm: families}, applied with $K=m^2$, ensures that
$$
|\lambda_1|
\le
\sigma_{T_{1,n}}
\le
T_{1,n}-m^2\nu_{1,n}
\le
\sum_{j=1}^{r}c_j-r\nu_{1,n}
$$
for all sufficiently large $n$. Lemma~\ref{lm: strips}, applied with $k=r$, $\mu=\nu_{1,n}$ and $s=1$, shows that the rectangles corresponding to $\lambda_1$ can be packed into the remaining strips. At the end, fill every remaining unit of area with a square $1\times1$. The resulting multiset is a partition of the rectangle $m\times n$.

The construction is injective. The added rectangles are squares $1\times1$, while every selected rectangle has its larger side at least $2$. After deleting the unit squares, every remaining rectangle has the canonical form $a\times b$ with $a\le b$, hence its smaller side determines whether the rectangle comes from $\lambda_1$ (smaller side equal to $1$) or from $\lambda_a$ (smaller side equal to $a\geq2$), while its larger side recovers the corresponding part. Therefore, the selected partitions $\lambda_1,\lambda_2,\ldots,\lambda_m$ can be recovered from the resulting multiset. Since the choices are independent,
$$
p(m,n)
\geq
|\mathcal A(2,T_{1,n})|
\prod_{a=2}^{m}|\mathcal A(a,T_{a,n})|.
$$
Lemma~\ref{lm: families} gives
$$
\log p(m,n)
\geq
\pi\sqrt{\tfrac23}\sum_{a=1}^{m}\sqrt{T_{a,n}}
+O(\log n).
$$
Since $0\le\tfrac{mn}{H_ma^2}-T_{a,n}<1$, we have
$$
\sqrt{\frac{mn}{H_ma^2}}-\sqrt{T_{a,n}}
\le
\frac{1}{\sqrt{T_{a,n}}}
=
O\left(n^{-1/2}\right),
$$
and therefore
$$
\sum_{a=1}^{m}\sqrt{T_{a,n}}
=
\sqrt{\frac{mn}{H_m}}\sum_{a=1}^{m}\frac1a
+O\left(n^{-1/2}\right)
=
\sqrt{mH_mn}
+O\left(n^{-1/2}\right).
$$
Consequently,
$$
\log p(m,n)
\geq
\pi\sqrt{\tfrac{2mH_m}{3}}\sqrt n
+O(\log n).
$$
That lower bound together with the upper estimate \eqref{eq: UB} completes the proof for $m\geq2$.

For $m=1$ the statement follows from Equality \eqref{eq: HR} and the fact that $p(1,n)=p(n)$. Finally, the upper bound $p(n)\le\exp\left(\pi\sqrt{2n/3}\right)$ is a direct consequence of the argument presented in the first part of the proof.
\end{proof}

\begin{re}
Figure~\ref{fig: columns} illustrates the construction from the proof of the lower bound for $m=5$. In such a setting, $q_2=2$ and $q_3=q_4=q_5=1$. The hatched strips receive the rectangles corresponding to $\lambda_2,\lambda_3,\lambda_4,\lambda_5$. The white strips of height $1$, i.e., the strips at the tops of the columns, the strips obtained from the unused parts of the hatched strips, and the strips obtained from the rectangle to the right of all the columns, receive the rectangles corresponding to $\lambda_1$ and the unit squares. The widths of the columns are drawn in the proportions $c_{a,n}:c_{a',n}$ determined by the main terms $\tfrac{T_{a,n}}{q_a}=\tfrac{mn}{H_ma^2q_a}+O(\sqrt n\log n)$. The remaining rectangle on the right occupies approximately the fraction $1-\theta_5\approx0.2585$ of the width $n$.
\end{re}

\begin{figure}[ht]
\centering
\begin{tikzpicture}[x=1cm,y=1cm,
  zone/.style={draw,line width=.55pt,align=center,font=\small},
  hatched/.style={pattern={Lines[distance=5pt,angle=45,line width=.3pt]},pattern color=black!40},
  lab/.style={font=\small,fill=white,fill opacity=.85,text opacity=1,inner sep=2pt}]
\def\u{0.62}
\def\xA{0}\def\xB{3.6}\def\xC{6.8}\def\xD{8.6}\def\xE{9.75}\def\xF{13}
\filldraw[zone,hatched] (\xA,0) rectangle (\xB,{2*\u});
\filldraw[zone,hatched] (\xA,{2*\u}) rectangle (\xB,{4*\u});
\filldraw[zone,fill=white] (\xA,{4*\u}) rectangle (\xB,{5*\u});
\node[lab] at ({(\xA+\xB)/2},{\u}) {$2\times c_{2,n}$};
\node[lab] at ({(\xA+\xB)/2},{3*\u}) {$2\times c_{2,n}$};
\filldraw[zone,hatched] (\xB,0) rectangle (\xC,{3*\u});
\filldraw[zone,fill=white] (\xB,{3*\u}) rectangle (\xC,{4*\u});
\filldraw[zone,fill=white] (\xB,{4*\u}) rectangle (\xC,{5*\u});
\node[lab] at ({(\xB+\xC)/2},{1.5*\u}) {$3\times c_{3,n}$};
\filldraw[zone,hatched] (\xC,0) rectangle (\xD,{4*\u});
\filldraw[zone,fill=white] (\xC,{4*\u}) rectangle (\xD,{5*\u});
\node[lab] at ({(\xC+\xD)/2},{2*\u}) {$4\times c_{4,n}$};
\filldraw[zone,hatched] (\xD,0) rectangle (\xE,{5*\u});
\node[lab,align=center] at ({(\xD+\xE)/2},{2.5*\u}) {$5\times$\\ $c_{5,n}$};
\filldraw[zone,fill=white] (\xE,0) rectangle (\xF,{\u});
\filldraw[zone,fill=white] (\xE,{\u}) rectangle (\xF,{2*\u});
\filldraw[zone,fill=white] (\xE,{2*\u}) rectangle (\xF,{3*\u});
\filldraw[zone,fill=white] (\xE,{3*\u}) rectangle (\xF,{4*\u});
\filldraw[zone,fill=white] (\xE,{4*\u}) rectangle (\xF,{5*\u});
\draw[line width=1pt] (\xA,0) rectangle (\xF,{5*\u});
\node[font=\footnotesize] at (-0.45,{2.5*\u}) {$m$};
\node[font=\footnotesize] at ({(\xA+\xF)/2},-0.45) {$n$};
\end{tikzpicture}
\caption{The decomposition of the rectangle $m\times n$ used in the proof of the lower bound of Theorem~\ref{thm: main} for $m=5$. The hatched strips are tiled by the rectangles coming from $\lambda_2,\lambda_3,\lambda_4,\lambda_5$, and the white strips of height $1$ receive the rectangles coming from $\lambda_1$ together with the unit squares.}
\label{fig: columns}
\end{figure}
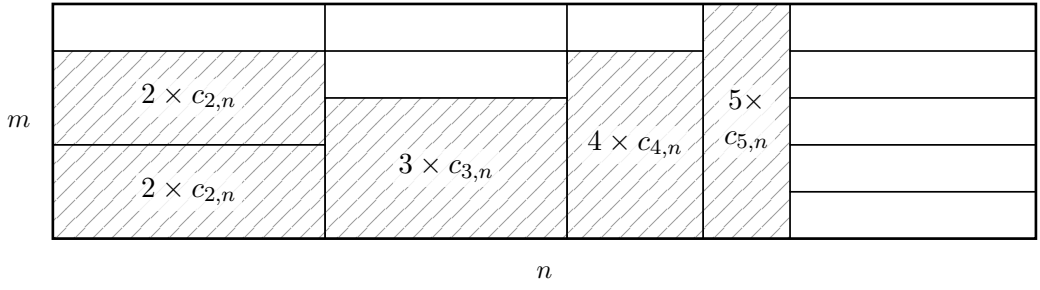

\begin{re}
For $m\le3$, the known asymptotic formulae for $p(m,n)$ exhibit a polynomial factor in front of the exponential term; see \cite{GVZ, GZ}. Determining this factor for a fixed $m\geq4$ remains an open problem.
\end{re}

\section*{Acknowledgments}
K. G. was supported by the National Science Center grant no.~2024/53/N/ST1/01538.

\end{document}